\documentclass[11pt]{amsart}

\usepackage{amssymb, amsmath, amsfonts, amsthm}
\usepackage{amsfonts}
\usepackage{tabularx,booktabs}
\usepackage{amsmath,tikz}
\usepackage{amsfonts}
\usepackage{todonotes}

\usepackage{ esint }
\usepackage[abbrev]{amsrefs}

\usepackage{amscd}
\usepackage{amsthm}
\usepackage{esint}
\usepackage{amssymb}
\usepackage{latexsym}
\usepackage{eufrak}
\usepackage{euscript}
\usepackage{epsfig}
\usepackage{array}
\usepackage{enumerate}
\usepackage{dsfont}
\usepackage{color}
\usepackage{wasysym}
\usepackage{hyperref}
\usepackage{pdfsync}
\usepackage[linesnumbered,boxed,vlined,ruled]{algorithm2e}

\newcommand{\ea}{\end{eqnarray}}

\newcommand{\qe}{\end{equation}}
\newcommand{\R}{{\mathbb R}}
\newcommand{\N}{{\mathbb N}}

\def \d {\partial}

\def \eps {\varepsilon}

\DeclareMathOperator{\Ric}{Ric}

\DeclareMathOperator{\vol}{vol}
\DeclareMathOperator{\Vol}{vol}
\DeclareMathOperator{\Hess}{Hess}

\DeclareMathOperator{\II}{II}
\DeclareMathOperator{\diam}{diam}

\DeclareMathOperator{\spec}{spec}
\DeclareMathOperator{\Sec}{Sec}
\DeclareMathOperator{\esssup}{esssup}

\newcommand{\f}[2]{\frac{#1}{#2}}

\newcommand{\drm}{\mathrm{d}}
\newcommand{\dvol}{\mathrm{dvol}}
\newcommand{\cA}{\mathcal{A}}

\newcommand{\Hmm}[1]{\leavevmode{\marginpar{\tiny%
$\hbox to 0mm{\hspace*{-0.5mm}$\leftarrow$\hss}%
\vcenter{\vrule depth 0.1mm height 0.1mm width \the\marginparwidth}
\hbox to 0mm{\hss$\rightarrow$\hspace*{-0.5mm}}$\\\relax\raggedright #1}}}

\theoremstyle{definition}
\newtheorem{theorem}{Theorem}[section]
\newtheorem{definition}[theorem]{Definition}

\newtheorem{lemma}[theorem]{Lemma}
\newtheorem{remark}[theorem]{Remark}

\newtheorem{proposition}[theorem]{Proposition}

\numberwithin{equation}{section}
\begin{document}

\title[Integral curvature and mass gaps]{Integral Ricci curvature and the mass gap of Dirichlet Laplacians on domains}

\author{Xavier Ramos Oliv\'{e}}
\email{\href{mailto:xramosolive@wpi.edu}{xramosolive@wpi.edu}}
\address{Xavier Ramos Oliv\'{e}: Department of Mathematical Sciences, Worcester Polytechnic Institute, Worcester, MA 01609}

\author{Christian Rose}
\email{\href{mailto:christian.rose@uni-potsdam.de}{christian.rose@uni-potsdam.de}}
\address{Christian Rose: Institute of Mathematics, University of Potsdam, Germany}

\author{Lili Wang}
\email{\href{mailto:lilyecnu@outlook.com}{lilyecnu@outlook.com}}
\address{Lili Wang: College of Mathematics and Information, FJKLMAA, Fujian Normal University, Fuzhou, 350108, China}

\author{Guofang Wei}
\email{\href{mailto:wei@math.ucsb.edu}{wei@math.ucsb.edu}}
\address{Guofang Wei: Department of Mathematics\\
         University of California\\
         Santa Barbara, CA 93106}

\thanks{ G.W. is partially supported by NSF
DMS 1811558, 2104704 }

\begin{abstract}
We obtain a fundamental gap estimate for classes of bounded domains with quantitative control on the boundary in a complete manifold with integral bounds on the negative part of the Ricci curvature. This extends the result of \cite{Oden-Sung-Wang99} to $L^p$-Ricci curvature assumptions, $p>n/2$. To achieve our result, it is shown that the domains under consideration are John domains, what enables us to obtain an estimate on the first nonzero Neumann eigenvalue, which is of independent interest. 

\bigskip
\noindent \textbf{Keywords.}  eigenvalue estimate, spectral gap, mass gap, integral Ricci curvature\\
\end{abstract}

\maketitle
\section{Introduction}
Consider a complete Riemannian manifold $M$ of dimension $n\in\N$, $n\geq 2$, and $\Omega\subset M$ open such that $\overline\Omega$ is a smooth compact manifold of dimension $n$ with smooth boundary $\partial\Omega$. The Dirichlet Laplacian $-\Delta\geq 0$ 
has discrete spectrum consisting of an unbounded increasing sequence of positive reals $\lambda_i(\Omega)$, $i\in\N$,  which can be arranged as
\[
0<\lambda_1(\Omega)<\lambda_2(\Omega)\leq \lambda_3(\Omega)\leq \ldots\to \infty 
\]
counting multiplicities. For fixed $\Omega\subset M$ as above, the difference
\[
\Gamma(\Omega) := \lambda_2(\Omega)-\lambda_1(\Omega)>0
\]
is called the \emph{mass gap} (or \emph{fundamental gap}) of the Dirichlet Laplacian on $\Omega$. In quantum mechanics, the mass gap represents the energy needed to jump from the ground state to the next lowest energy state. The subject has a long history and has been a very active area of research recently, see e.g. the survey article \cite{DaiSetoWei-survey}. 

In the celebrated work \cite{andrewsclutterbuckgap}, Andrews and Clutterbuck proved the fundamental gap conjecture that for bounded and convex domain $\Omega\subset \R^n$, 
\[
\Gamma(\Omega)\geq 3\pi^2\diam (\Omega)^{-2}.
\]
This has been generalized to convex domains in  $\mathbb{S}^n$ 
in \cite{SetoWangWei16, HeWei20, DaiSetoWei18}, showing the same gap estimate. On the other hand,  it was proven in 
\cite{BourniClutterbuckNguyenStancuWeiWheeler2} that given any diameter there are convex domains $\Omega\subset\mathbb{H}^n$ with arbitrarily small fundamental gap.

For the lower bound estimate on the gap, the convexity condition is essential. In \cite{Cheng-Oden97}  a lower bound on $\Gamma(\Omega)$ was derived for domains $\Omega\subset\R^n$ assuming the so-called interior rolling $R$-ball condition for $\partial\Omega$, cf.~Definition~\ref{interior rolling}, in terms of bounds on the second fundamental form and volume, where  $\Omega$ is not necessarily convex. 
This is generalized to compact manifolds in \cite{Oden-Sung-Wang99}, 
showing that the gap is bounded from below in terms of uniform lower bounds on the Ricci curvature, the diameter, sectional curvature near the boundary, and the interior rolling $R$-ball condition. The construction in \cite{BourniClutterbuckNguyenStancuWeiWheeler2} shows that the rolling $R$-ball condition is necessary there and for general manifolds the $R$-rolling ball condition is more suitable than the convexity condition in some sense.

The present paper provides a gap estimate for suitable subsets of manifolds assuming only integral bounds on the negative part of the Ricci curvature. 
Such integral curvature conditions gathered a lot of attention during the last decades because in contrast to lower Ricci curvature pointwise bounds they are more stable under perturbations of the metric, see, e.g., \cite{Aubry-07,Chen-21,ChenH-21,PetersenWei-97, PetersenWei-01, DaiWeiZhang-18,Gallot-88, Rose-17, Rose-17a, OlivePostRose-20,Wang-20,ZhangZhu-17} and the references therein.

We fix some notation in order to state our main result.
Denote by $B_r(x)$ the geodesic ball centered at $x\in M$ with radius $r>0$. Furthermore, for $\Omega\subset M$ with smooth boundary $\partial\Omega$, we let $\II:=\II_{\partial\Omega}$ be its second fundamental form with respect to the inward pointing normal. Moreover, we denote by $\Sec$ the sectional curvatures.

Define $$\rho\colon M\to \R, \quad x\mapsto \min \spec(\Ric_x),$$
where $\Ric$ is considered as pointwise endomorphism on the tangent bundle. For $x\in\R$ we denote $x_-:=\max\{0,-x\}$.
Let $p>n/2$ and define for $x\in M$ and $r>0$
$$\kappa(x,p,r):=r^2\left(\fint_{B_r(x)}\rho_-^p\right)^\f{1}{p},$$
measuring the $L^p$-mean of the negative part of the Ricci curvature in a ball with respect to the Riemannian volume form $\dvol$.
It is convenient to work with the scaling invariant curvature quantity introduced in \cite{PetersenWei-97}
\begin{align*}\label{kappa-cur-invari}
\kappa(p,r)=\sup\limits_{x\in M}\kappa(x,p,r).
\end{align*}

In this paper we deal with the following class of subsets.
\begin{definition} \label{RHK-regular} Let $M$ be a Riemannian manifold of dimension $n\in\N$, $n\geq 2$, without boundary, $H,K>0$, and $R>0$. A subset $\Omega\subset M$ is called \emph{$(R,H,K)$-regular} if 
\begin{itemize}
\item[-] $\Omega$ is open, bounded, and connected,
    \item[-] $\overline{\Omega}\neq M$ is a smooth Riemannian manifold of dimension $n$ with smooth boundary $\partial\Omega$,
    \item[-] $\partial\Omega$ satisfies $\II_{\partial\Omega}\leq H$ and the interior rolling $R$-ball condition, cf.~Definition~\ref{interior rolling},
    \item[-] $\vert\Sec\vert\leq K$ in the inner $R$-tubular neighborhood $T(\partial\Omega,R):=\{x\in \Omega| d(x,\partial\Omega)<R\}$.
\end{itemize}
\end{definition}

Our main result can now be stated as follows.
\begin{theorem}\label{main1}
Let $2p>n\geq 2$, 
$D>0$, $H,K >  0$. There exists $R_0=R_0(K,H)>0$ such that the following holds: for any $0<R\leq R_0$, there exist explicitly computable constants 
$C=C(n, p, D, R, H,K)>0$ and $\eps=\eps(n,p)>0$ such that if $M$ is a complete manifold of dimension $n$ with $\kappa(p,D)\leq \eps$, then for any $(R,H,K)$-regular $\Omega\subset M$ with $\diam \Omega \leq D$
we have
\[
\Gamma(\Omega)\geq C(n, p, D, R, H,K).
\]
\end{theorem}

General criteria for a quantitative lower bound on the gap for compact manifolds have been provided in \cite{Oden-Sung-Wang99}, cf.~Theorem~\ref{OSW-gap-est}.
The key ingredients are volume doubling, a Harnack type  inequality for the first Dirichlet eigenfunction, weak Neumann-Poincar{\'e} inequalities on interior balls, and Neumann eigenvalue estimates for certain subsets of the manifold. We will show that all of these conditions are satisfied for $(R,H,K)$-regular domains with integral Ricci curvature bounds, which in turn yields our result. The volume doubling condition follows directly from \cite{PetersenWei-01, Ramos-Olive19}, and the Harnack inequality for the first Dirichlet eigenfunction follows from a combination of the results in \cite{DaiWeiZhang-18,Oden-Sung-Wang99,PetersenWei-97}, and \cite{PetersenWei-01}.
The crucial step to obtain the main result proved in the present article is the following new Neumann eigenvalue estimate for $(R,H,K)$-regular domains.

\begin{proposition} \label{Neumann-eigen}
Let $2p>n\geq 2$, $D>0$, and $K,H> 0$.
There exist explicitly computable $\epsilon=\epsilon(n,p)>0$, $R_0=R_0(K,H)>0$
such that the following holds: for any $R\in(0,R_0]$ there is an explicitly computable $C=C(n,R,D,H,K)>0$ such that for any Riemannian manifold of dimension $n$ with
$ \kappa_M(p,D)\leq \epsilon$,
any $(R,H,K)$-regular domain $\Omega\subset M$ with $\diam \Omega\leq D$, the first (non-zero) Neumann eigenvalue $\eta_1(\Omega)$ of $\Omega$ satisfies
\begin{equation*}
\eta_1(\Omega)\geq C.  \label{Neum-eigen}
\end{equation*}
\end{proposition}

To achieve this we prove that $(R,H,K)$-regular sets are John domains in the sense of 
\cite{HajlaszKoskela-00}, cf.~Definition~\ref{def:John}, which is of independent interest. 
Based on techniques in \cite{DaiWeiZhang-18} we derive weak Neumann-Poincar{\'e} inequalities for balls.  
The desired Neumann eigenvalue estimate then follows from the main results in \cite{HajlaszKoskela-00}.
Prop.~\ref{Neum-eigen} also generalizes the recently appeared eigenvalue estimate in  \cite[Corollary~1.5]{OlivePostRose-20} obtained by completely different techniques. 

The structure of this paper is as follows: in Section~\ref{preliminary} we fix notation and recall the Laplace and volume comparison and Sobolev constant estimates for integral Ricci curvature. 
The essential John domain property of $(R,H,K)$-regular domains will be proven in Section~\ref{section:poincare} as well as the weak Neumann-Poincar\'e inequalities in balls, 
yielding Prop.~\ref{Neum-eigen}. We adapt the local Harnack inequalities for the first Dirichlet eigenfunction from \cite{Oden-Sung-Wang99} in Section~\ref{local-Harnack-inequality} to $(R,H,K)$-regular domains under integral curvature conditions, and prove Theorem~\ref{loc-grad-est}. To derive Theorem~\ref{main1}, global Harnack inequalities for the first Dirichlet eigenfunction are needed, which are derived in Appendix~\ref{Global-Harnack-inequality}.
\\

\textbf{Acknowledgements.}
C.R.~wants to thank G.W.~and UCSB for providing a nice environment during his stay at UCSB, where parts of this work had been done. We also thank Zhenlei Zhang for sharing his private notes on isoperimetric constants.
Part of this work was done while L.W. was visiting UCSB during 2017. She would like to thank UCSB for providing great environment for research, Professor Y.~Zheng for his constant support and thank Professor J.~P.~Wang for answering her questions patiently. L.W. is supported by NSFC Grant no. 11961131001. 

\section{Preliminaries on integral Ricci curvature}\label{preliminary}
For $p\in[1,\infty)$, a measurable function $f\colon M\to\R$, and any geodesic ball $B\subset M$ 
we denote as usual 
\[
\|f\|_{p,B}=\left(\int_{B}|f|^p\right)^\f{1}{p},\ \ \|f\|_{p,B}^*=\left(\f{1}{\vol\left(B\right)}\int_{B}|f|^p\right)^\f{1}{p},
\quad\text{and}\quad 
\Vert f\Vert_{\infty,B}=\esssup_B f.
\]
If not explicitly stated differently we integrate w.r.t.~the Riemannian volume measure $\dvol$.
We will need the Laplacian and volume comparison estimate for integral curvature from \cite{PetersenWei-97, PetersenWei-01}. 
Given $x\in M$, let $d(y)=d(x,y)$ be the distance function and
\begin{align}\begin{split}\label{psi}
\psi(y):=\left(\Delta d-\f{n-1}{d}\right)_+.
\end{split}\end{align}
The classical Laplacian comparison states that $\Ric_M\geq 0$ implies $\Delta d\leq \f{n-1}{d}$. 
The integral curvature version of the Laplacian and volume comparison estimates we will use here are as follows.
\begin{theorem}\cite{PetersenWei-97}\label{psi-comp}
Let $p>\f{n}{2}$, $x\in M$, $r>0$. We have
\begin{align*}\begin{split}\label{lap-com-int}
\|\psi\|_{2p, B_r(x)}^*\leq C(n,p)\left(\|\rho_-\|_{p,B_r(x)}^*\right)^\f{1}{2}\leq C(n,p)r^{-1}\kappa(p,r)^\f{1}{2},
\end{split}\end{align*}
where
\[
C(n,p)=\left(\f{(n-1)(2p-1)}{2p-n}\right)^\f{1}{2}.
\]
\end{theorem}
 \begin{theorem}[{\cite{PetersenWei-97, PetersenWei-01}}]\label{volume-comp}
Let 
$x\in M$, $p>\frac{n}{2}$. There exists $\epsilon_0=\epsilon_0(n,p)>0$  such that if  $\kappa(x,p,r_0)\leq \epsilon_0$ we have for all $0<r\leq s\leq r_0$
\begin{align*}
\frac{\vol(B_s(x))}{\vol(B_r(x))}\leq 2\left(\frac sr\right)^n .
\end{align*}

\end{theorem}

\begin{remark}(cf.~\cite[Remark 2.2]{DaiWeiZhang-18} \cite[Section 2.3]{PetersenWei-01})
If $\kappa(x,p,r_2)\leq \eps_0$ for $\eps_0$ as above, 
Theorem~\ref{volume-comp} implies for all
for all $x\in M$ and $0<r_1\leq r_2$
\begin{align}\begin{split}\label{scal-inv-contr-eachother1}
\kappa(x,p,r_1)\leq 2^\f{1}{p}\left(\f{r_1}{r_2}\right)^{2-\f{n}{p}}\kappa(x,p,r_2)\leq 2^\f{1}{p}\kappa(x,p,r_2).
\end{split}\end{align}
Hence, $\kappa(x,p,r_1)\leq \eps_0$ if $r_1\leq 2^{\f{1}{n-2p}}r_2$. If $\kappa(p,r_1)\leq \eps_0$, then for all $r_2\geq r_1$, we have 
\begin{align}\begin{split}\label{scal-inv-contr-eachother2}
\kappa(p,r_2)\leq 2^\f{n+1}{p}\left(\f{r_2}{r_1}\right)^2\kappa(p,r_1).
\end{split}\end{align}
Hence when $\kappa(p, r)$ is small for some $r$, it gives control on $\kappa(p, r)$ for all $r$.
Note that for compact manifolds with diameter $D$ there is no restriction between working with
$\kappa(p, 1)$ and the more global constant $\kappa(p, D)$.
\end{remark}

In this paper, we consider bounded and connected domains $\Omega\subset M$ such that $\overline\Omega\subset M$, $\overline\Omega\neq M$, is a compact Riemannian manifold with not necessarily convex smooth boundary $\partial\Omega$. We assume the following boundary regularity instead.
\begin{definition}\label{interior rolling}
Let $\Omega$ be a compact manifold with smooth boundary $\partial\Omega$ and $R>0$. $\d \Omega$ satisfies the \textit{interior rolling $R$-ball condition} if for each point $y\in \d \Omega$ there is a geodesic ball $B_R(q)\subset \Omega$ with $\overline{B_R(q)}\cap \d \Omega=\{y\}$. $R$ is called the \textit{interior rolling ball radius}.
\end{definition}

The following relative volume comparison for concentric geodesic balls of compact submanifolds with smooth boundary satisfying the interior rolling $R$-ball condition and integral Ricci curvature bounds has been obtained in \cite{Ramos-Olive19}. 
\begin{lemma}[{\cite[Lemma~3.3]{Ramos-Olive19}}]\label{vol-comp-sub} Let $D>R>0$, $n\geq 2$, $p>\frac{n}{2}$, 
and $\Omega\subset M^n$ a domain with $\diam(\Omega)\leq D$ such that $\overline\Omega\neq M$ is a smooth manifold with boundary $\partial\Omega$ satisfying the interior rolling $R$-ball condition. There exists an $\eps_0=\eps_0(n,p)>0$ such that if  $\kappa(p, D)\leq \epsilon_0$, then we have  for all $x\in\Omega$,  all $ 0<r_1\leq r_2\leq D,$
\begin{align}\begin{split}
\frac{\vol(B^\Omega_{r_2}(x))}{\vol(B^\Omega_{r_1}(x))}\leq C_0\left(\frac{r_2}{r_1}\right)^n, \label{vol-double-sub}
\end{split}\end{align}
where $B^\Omega_r(x):=B_r(x)\cap\Omega$ and $C_0=2\left(\frac{2D}{R}\right)^n$.
\end{lemma}

Another key tool we will use is the local Sobolev constant estimate in \cite{DaiWeiZhang-18}. 
We denote by $C_s(\Omega)$ the normalized local Sobolev constant of $\Omega$, that is,
\begin{align}\begin{split}\label{normalized-Sobolev-const}
\left(\fint_\Omega f^\frac{2n}{n-2}\right)^\frac{n-2}{2n}\leq C_s(\Omega)\left(\fint_\Omega |\nabla f|^2\right)^\f{1}{2}, \ \ \forall f\in C_0^\infty(\Omega).
\end{split}\end{align}
Note that $C_s(\Omega)$ scales like diameter. Here is the estimate we will need in Section 4. 
\begin{theorem}[{\cite[Corollary 4.6]{DaiWeiZhang-18}}] \label{thm:local-Sobolev}
For $p > n/2,\ D>0$, there exists $\epsilon = \epsilon(n,p) >0$ such that if $M^n$ has $\kappa (p, D)\le \epsilon$,
then for any $x \in M$ with $\partial B_D(x) \not= \emptyset$, and any $0< r \le D$, the normalized local Sobolev constant of the ball $C_s(B_r(x))$ has the estimate
\begin{equation}
    \| f\|^*_{\frac{2n}{n-2}, B_r(x)} \leq C(n) r \, \|\nabla f\|^*_{2, B_r(x)}, \ \ \forall  \  f\in C_0^\infty(B_r(x)).  \label{local-Sobolev}
\end{equation}
\end{theorem}

\section{Neumann-Poincar\'e inequalities for John Domains}\label{section:poincare}

This section is devoted to the proof of Proposition~\ref{Neumann-eigen}. 
We show that $(R,H,K)-$regular domains defined in Definition~\ref{RHK-regular} are John domains defined below. 

\begin{definition}[{\cite[Page 39]{HajlaszKoskela-00}}]\label{def:John}
A bounded open subset $\Omega\subset M$ is called a \emph{John domain} if  there exist $x_0\in \Omega$ and $C_J >0$ such that for every $x\in \Omega$ there exists a curve $\gamma:[0,l]\rightarrow \Omega$ parametrized by arclength with $\gamma(0)=x$, $\gamma(l) =x_0$, and
\[{\rm dist}(\gamma(t), \partial \Omega)\geq C_J t.\]
\end{definition}
Clearly, any ball of a geodesic metric space is a John domain with $C_J=1$ by taking $x_0$ to be the center of the ball. 
John domains are a very general class of subsets of metric spaces possessing strong inclusion properties of function spaces.
A key property is that they satisfy the chain condition. Hence, to obtain a Sobolev inequality for the domain it is suffices to work 
 on balls inside the domain, see \cite[Chapter 9]{HajlaszKoskela-00} for details. 

 Here we prove that $(R,H,K)$-regular domains are John domains.
\begin{lemma}  \label{RHK-John} Let $\Omega\subset M$ be an $(R,H,K)-$regular domain with ${\rm diam}(\Omega)\leq D$. Then $\Omega$ is a John domain with $C_J=C(n,D,R,H,K)>0$.
\end{lemma}
\begin{proof} For $\delta >0$ denote
\[\Omega_{\delta} = \{x\in \Omega \colon {\rm dist}(x,\partial \Omega)> \delta\}\]
and $\widetilde{D}_\delta= {\rm diam}(\Omega_{\delta})$ the intrinsic diameter of $\Omega_{\delta}$. Since $\Omega$ is an $(R,H,K)-$regular domain,  the proof of   \cite[Lemma~3.2.7, Page~60]{Oden94} gives for $\delta < R/2$
\begin{equation}
\widetilde{D}_\delta = \diam(\Omega_\delta)\leq (C_{K,H,\delta})^{n-1}D,  \label{diam-est}
\end{equation}
where
\[
C_{K,H,\delta}:=2\left(\frac{K}{H}+\frac{H}{K}\right)\cosh(2\sqrt{K}\delta).
\]
In particular, we have
 $\tilde{D}_{R/4} \leq (C_{K,H,R/4})^{n-1}D$.
Moreover, $\Omega_{R/4}\not = \varnothing$ since $\Omega$ satisfies the rolling $R-$ball condition. Pick any fixed $x_0\in \Omega_{R/4}$ and $x\in \Omega$.
If $x\in \Omega_{R/4}$, let $\gamma:[0,l]\rightarrow \Omega_{R/4}$ be a minimizing normal
geodesic in $\Omega_{R/4}$ with $\gamma(0)=x$ and $\gamma(l)=x_0$. Then $l$, the intrinsic distance between $x_0$ and $x$ in $\Omega_{R/4}$, will satisfy $l\leq \widetilde{D}_{R/4}$. For $C= \frac{R}{4\widetilde{D}_{R/4}}$ we have
\[{\rm dist}(\gamma(t), \partial \Omega) \geq \frac{R}{4} \geq C\widetilde{D}_{R/4} \geq Ct. \]
If $x\in \Omega \setminus \Omega_{R/4}$, let $p\in \partial \Omega$ be such that ${\rm dist}(x,p)={\rm dist}(x,\partial \Omega)$, let $q\in \Omega$ be the center of 
a rolling $R$-ball 
such that $\overline{B_R(q)}\cap \partial \Omega =\{p\}$, and let $\gamma_1$ be a distance minimizing normal geodesic from $p$ to $q$. Since $\gamma_1$ goes through $x$, choose the parametrization so that $\gamma_1(0) = x$ and $\gamma_1(dist(x,q)) = q$. Notice that $q\in \overline{\Omega_{R/4}}$, since ${\rm dist}(q,\partial \Omega)=R$. Let $\gamma_2$ be a minimizing 
normal 
geodesic in $\Omega_{R/4}$ joining $q=\gamma_2(0)$ and $x_0 = \gamma_2(l)$.  Then consider the curve $\gamma:[0, l+{\rm dist}(x,q)]\rightarrow \Omega$ defined by
\[\gamma(t) = \begin{cases}
                \gamma_1(t) & \text{if } 0\leq t\leq {\rm dist}(x,q),\\
                \gamma_2(t-{\rm dist}(x,q)) & \text{if } {\rm dist}(x,q)< t\leq l+{\rm dist}(x,q).
            \end{cases}
\]
For the choice $C= \frac{R}{4\widetilde{D}_{R/4}+R}$ we need to consider the following  two cases.
\vskip 1em
\emph{Case 1:} $t<{\rm dist}(x,q)$
\vskip 1em
In this case, we have
\[{\rm dist}(\gamma(t),\partial \Omega) = {\rm dist}(\gamma_1(t),\partial \Omega) = {\rm dist}(\gamma(t),p) = {\rm dist}(x,p)+t\geq t \geq Ct. \]
\vskip 1em
\emph{Case 2:} ${\rm dist}(x,q)\leq t \leq {\rm dist}(x,q)+l$
\vskip 1em
In this situation, we have $t\leq R/4+\widetilde{D}_{R/4}$. Since $\gamma_2$ is a curve in $\Omega_{R/4}$, we have
\[{\rm dist}(\gamma(t),\partial \Omega) = {\rm dist}(\gamma_2(t-{\rm dist}(x,q)),\partial \Omega) \geq \frac{R}{4} = C(\widetilde{D}_{R/4}+R/4) \geq Ct.\]
Hence, we have shown that $\Omega$ is a $C$-John domain for $C = \frac{R}{4\widetilde{D}_{R/4}+R}$. We conclude the proof by noting $\widetilde{D}_{R/4}\leq C(n,D,R,H,K)$.
\end{proof}

According to \cite[Theorem 9.7]{HajlaszKoskela-00}, to prove Proposition~\ref{Neumann-eigen} it now suffices to show the following weak Neumann-Poincar\'e inequalities for all balls $B_r$ in an $(R,H,K)$-regular $\Omega$ with $B_r\cap\partial\Omega=\emptyset$:
\begin{equation}
    \|f - f_{B_{\frac r{32}}}\|^*_{1,  B_{\frac r{32}}} \le C(n) r \|\nabla f\|^*_{1, B_r},\quad f \in W^{1,2} (\Omega), \label{n-L2}
\end{equation}
where $f_{B_{s}}$ denotes the average of $f$ on $B_{s}$. 

We use the technique in \cite{DaiWeiZhang-18} for obtaining the Dirichlet Poincar\'e inequalities for balls (see Theorem~\ref{thm:local-Sobolev}) to prove the above estimate. Our start point is the following weak Cheeger’s constant estimate with an error \cite[Corollary 4.3]{DaiWeiZhang-18}.

\begin{lemma} \label{lemma2} If  $B_{2r}(x) \subset \Omega$, $B_{2r}\cap\Omega=\emptyset$, and $H$ is a hypersurface dividing $B_{2r}(x)$ into two parts $\Omega_1$ and $\Omega_2$, then we have 
\begin{align}\begin{split}\label{area-est}
\min\left\{\Vol(\Omega_1\cap B_r(x)),\Vol(\Omega_2\cap B_r(x))\right\}\leq 2^{n+1}r \cA(H\cap B_{2r}(x))+ 2^n \vol(B_{2r}(x)) \kappa(x,p,2r)^\frac{1}{2}.
\end{split}\end{align}
\end{lemma}

If $\kappa(x,p,2r) \le \epsilon_0(n,p)$, where $\epsilon_0(n,p)$ is the constant in Theorem~\ref{volume-comp}, then \[ \vol(B_{2r}(x))  \le 2^{n+1} \vol(B_{r}(x)).   \]
Assuming additionally $\kappa(x,p,2r) \le  2^{-2(4n+5)}$ and 
\begin{equation}\label{divide-volume-contr}
\min(\Vol(\Omega_1\cap B_r(x)),\Vol(\Omega_2\cap B_r(x)))\geq\frac {1}{2^{2n+3}}\vol\left(B_r(x)\right), 
\end{equation}
we get
\begin{equation}\label{area-contr-volume}
\vol(B_r(x)) \leq 2^{3n+5}r\cA(H\cap B_{2r}(x)).
\end{equation}

Let $\Omega\subset M$ be a bounded domain and $\Omega'\subset\Omega$ a subdomain. For convenience we consider  the relative isoperimetric constant of $\Omega'$ relative to $\Omega$ as
\[
I_R (\Omega',\Omega):=\sup \frac{\min(\Vol(\Omega_1'),\Vol(\Omega_2'))}{\cA(H)},
\]
where $H$ ranges over all hypersurfaces in $\Omega$ dividing $\Omega$ into two parts $\Omega_1$ and $\Omega_2$, and $\Omega_i'=\Omega_i\cap \Omega'$, $i=1,2$.
\begin{proposition}\label{prop:zhang} 

 Let $r>0$ and $x\in \Omega$ be such that $B_r(x)\cap \partial \Omega = \emptyset$. There exists $\epsilon(n,p)>0$ such that if $\kappa(p,r)\leq\epsilon$, then \begin{equation*} I_R( B_{r/32}(x), B_r(x)) \leq C(n)r,  \label{R-Iso-est}
 \end{equation*}
where $C(n) =2^{4n+2}5^n$. 

\end{proposition}
\begin{proof}

Let $H$ be any hypersurface dividing $B_r(x)$ into two subsets $\Omega_1$, $\Omega_2$. Define $\Omega_i':= \Omega_i\cap B_{r/32}(x)$ for $i=1,2$. We assume w.l.o.g.~$\Vol(\Omega_1') \leq \Vol (\Omega_2')$. This implies $\Vol(\Omega_2')\geq \frac{1}{2}\Vol{B_{r/32}(x)}$. For any $y\in \Omega_1'$, consider
\[
r_y:=\inf\left\{s>0\colon \Vol(B_s(y)\cap\Omega_2)\geq \frac{1}{2^{2n+3}} \vol(B_s(y))\right\}.
\]
We observe that the set of radii on the right hand side is not empty. Indeed, for $s=r/16$, we have that $B_{r/32}(x)\subseteq B_{r/16}(y) \subseteq B_{r/8}(x)$, thus using Theorem \ref{volume-comp}
\[\Vol(B_{r/16}(y)\cap \Omega_2) \geq \Vol (\Omega_2') \geq \frac{1}{2} \Vol(B_{r/32}(x)) \geq \frac{1}{2^{2n+3}}\Vol(B_{r/8}(x)) \geq \frac{1}{2^{2n+3}}\Vol(B_{r/16}(y)).\]
Notice also that $r_y>0$, because the ratio $\frac{\Vol(B_{s}(y)\cap\Omega_2)}{\Vol(B_s(y)}$ is continuous in $s$ and becomes $0$ for $s$ small enough. Thus for any $y\in \Omega_1'$ we have that $0<r_y\leq \frac{r}{16}$.

The set $\{B_{2r_y}(y)\}_{y\in \Omega_1'}$ gives an open cover for $\Omega_1'$. Then, by the Vitali covering lemma, there exists a subfamily of disjoint balls $\displaystyle \{B_{2r_i}(y_i)\}_{i\in I}$, $r_i = r_{y_i}$, such that \[\Omega_1' \subseteq \bigcup_{i\in I} B_{10r_i}(y_i).\] Notice that, since $r_y\leq \frac{r}{16}$ and $y\in B_{r/32}(x)$, we have that $B_{10r_i}(y_i) \subseteq B_r(x)$. In particular, $B_{10r_i}(y_i) \cap \partial \Omega  = \emptyset$ and $10r_i\leq r$, so we can use Theorem \ref{volume-comp} on these balls to get
\[\Vol(\Omega_1') \leq \sum_{i\in I}\Vol (B_{10r_i}(y_i)) \leq 2\cdot 10^n \sum_{i\in I} \Vol(B_{r_i}(y_i)). \]
On the other hand, by definition $\Vol(B_{r_i}(y_i)\cap \Omega_2) = \frac{1}{2^{2n+3}}\vol(B_{r_i}(y_i))$, so $\Vol(B_{r_i}(y_i)\cap \Omega_1) = \left(1-\frac{1}{2^{2n+3}} \right) \vol(B_{r_i}(y_i)) > \Vol(B_{r_i}(y_i)\cap \Omega_2)$. Thus, choosing $\epsilon(n,p)$ small enough, we can use \eqref{area-contr-volume}. Since the balls $\{B_{2r_i}(y_i)\}_{i\in I}$ are disjoint, we have 
\[\cA(H) \geq \sum_{i\in I} \cA(H\cap B_{2r_i}(y_i)) \geq 2^{-(3n+5)} \sum_{i\in I} r_i^{-1}\Vol(B_{r_i}(y_i)).\]

Combining the two estimates together, we get
\[\frac{\Vol(\Omega_1')}{\cA(H)}\leq 2^{3n+5} (2\cdot 10^n) \frac{\sum_{i\in I} \Vol(B_{r_i}(y_i))}{\sum_{i\in I} r_i^{-1}\Vol(B_{r_i}(y_i))} \leq 2^{3n+6}10^{n} \sup_{i\in I} r_i \leq 2^{4n+2}5^n r.\]

\end{proof}

By the equivalence of the isoperimetric constant and the Sobolev constant (the same proof as \cite[Theorem 9.6]{Li12} applies to the weak version), we have 

\begin{equation}
 \|f - f_{B_{\frac r{32}}}\|_{1,  B_{\frac r{32}}} \le 2 \inf_{a \in \mathbb R}  \|f - a\|_{1,  B_{\frac r{32}}} \le  2 I_R( B_{r/32}(x), B_r(x)) \|\nabla f\|_{1, B_r}. 
\end{equation}
Proposition~\ref{R-Iso-est} and volume doubling, i.e., Theorem~\ref{volume-comp},  give (\ref{n-L2}) when $\kappa(p, r) \le \epsilon(n,p)$. 

Finally, we are ready to prove Proposition~\ref{Neumann-eigen}.
\begin{proof}[Proof of Proposition~\ref{Neumann-eigen}] By (\ref{scal-inv-contr-eachother1}), we have $\kappa (p, r) \le 2^{1/p} \kappa (p, D)$. Choosing $\epsilon(n,p)$ smaller we get (\ref{n-L2}) if $\kappa(p, D) \le \epsilon(n,p)$. 
Since we have volume doubling (\ref{vol-double-sub}), the weak Neumann-Poincar\'e inequality (\ref{n-L2}), and Lemma~\ref{RHK-John}, we can apply \cite[Theorem 9.7]{HajlaszKoskela-00} (for $p=1, s=n$) to obtain
\begin{equation*}
  \inf_{a \in \mathbb R}  \|f - a\|^*_{\tfrac{n}{n-1},  \Omega} \le C(n, C_0, C_J) \diam (\Omega) \, \|\nabla f\|^*_{1, \Omega}\quad f \in W^{1,2} (\Omega), 
\end{equation*}
where $C_0, C_J$ are constants from Lemmas~\ref{vol-comp-sub} and \ref{RHK-John}. 
H\"older's inequality yields \[  \|f - a\|^*_{1,  \Omega} \le  \|f - a\|^*_{\tfrac{n}{n-1},  \Omega}. 
\]
Hence \begin{equation}
  \inf_{a \in \mathbb R} \|f - a\|^*_{1,  \Omega} \le   C(n, C_0, C_J) \diam (\Omega) \, \|\nabla f\|^*_{1, \Omega}\quad f \in W^{1,2} (\Omega).  \label{l1-neumann}
\end{equation}
Applying Cheeger's inequality (see e.g. \cite[Page 92]{Li12}) gives the Neumann eigenvalue estimate (\ref{Neum-eigen}). 
\end{proof}

Now when $\Omega$ is a ball $B_r(x)$, then $C_J =1, R =r, D = 2r$, so $C_0 = 2\cdot 4^n$.  Therefore applying the  estimate (\ref{l1-neumann}) to the ball, we have 
\[ \inf_{a \in \mathbb R} \|f - a\|^*_{1,  B_{r}} \le  C(n)r \|\nabla f\|^*_{1, B_r}, \quad f\in W^{1,2}(B_r).
\]
Again applying Cheeger's inequality immediately gives the Neumann-Poincar\'e inequality for the ball, with the following explicit dependence which we will also need. 
\begin{equation}\label{weakNP}
 \|f - f_{B_{r}}\|_{2,  B_{r}} \le  C(n)r \|\nabla f\|_{2, B_r}, \quad f\in W^{1,2}(B_r).
\end{equation}
In fact we only need the weak $L^2$ version of the Neumann-Poincar\'e inequality. While it is automatic to get the $L^2$ version from the $L^1$ version for the usual Neumann-Poincar\'e inequality, this is not the case for the  weak version.  The estimate (\ref{weakNP}) has been pointed out in \cite[Remark 1.7]{DaiWeiZhang-PAMS}. 

\section{Local Harnack inequality for the first Dirichlet eigenfunction}\label{local-Harnack-inequality}
In this section, we will prove a local Harnack inequality for the first Dirichlet eigenfunction via a  gradient estimate. For pointwise curvature bounds, one can prove it using the maximum principle, see e.g. \cite[Theorem 6.1]{Li12}. 
For integral curvature conditions, the gradient estimate can be established via the Nash-Moser iteration. The essential tools are  the local Sobolev inequality  in Theorem~\ref{thm:local-Sobolev} 
and the Laplacian comparison estimate in Theorem~\ref{psi-comp}.

We have the following gradient estimate for the first eigenfunction of ball depending on integral Ricci curvature bounds. 
\begin{theorem}
\label{grad-esti-u}
Let $p>n/2$, $x\in M$, $r>0$, and $\lambda\geq 0$. If $u$ is a positive solution of 
\[
\Delta u=-\lambda u
\]
on $B_r(x)$, then
\begin{align*}
\sup_{B_{\f{r}{2}}(x)}|\nabla \ln u|^2
&\leq C(n,p)\left(r^{-2}\f{\vol\left(B_r(x)\right)}{\vol\left(B_{\f{7}{8}r}(x)\right)}+\lambda \right)
\\
&\quad\cdot
\left[r^{-2}C_s^2\left(B_r(x)\right)\left(1+r^{-2}C_s^2\kappa(p,r)\right)+\left(r^{-2}C_s^2\left(B_r(x)\right)\kappa(p,r)\right)^\frac{2p}{2p-n}\right]^{\f{n^2}{2}}.
\end{align*}
\end{theorem}
The proof is a modification of the argument in \cite[Theorem 5.2]{DaiWeiZhang-18}.
\begin{proof}
Let $h=\ln u$ and $v=|\nabla h|^2+\lambda$, such that $\Delta h=-v$. By scaling we assume that $r=1$. We infer from the Bochner formula
\begin{align}
\label{Bochner formula1}
\Delta v
=\Delta|\nabla h|^2
=2|\Hess h|^2+2\langle \nabla h, \nabla \Delta h \rangle+2\Ric(\nabla h, \nabla h)
\geq \frac{2}{n}v^2-2\langle \nabla h, \nabla v\rangle-2\rho_-v.
\end{align}
We abbreviate $B_r:=B_r(x)$. All integrals below are on $B_1$ which we omit. 
Given $\eta\in C^\infty_0(B_1)$ and $l>1$,  integration by parts yields
\begin{align}
\label{grad-squar-term}
\int|\nabla(\eta v^l)|^2
&=-\int\left(\eta v^l\right)\left[v^l\Delta\eta+2\langle \nabla\eta, \nabla v^l\rangle+ \eta\Delta v^l\right]\nonumber\\
&=\int v^{2l}\left(-\eta\Delta\eta\right)-2\int v^l\langle \nabla\eta, \eta \nabla v^l\rangle
-\int \eta^2v^l\left(lv^{l-1}\Delta v +l(l-1)v^{l-2}|\nabla v|^2\right)\nonumber\\
&=\int v^{2l}\left(-\eta\Delta\eta+2|\nabla \eta|^2\right)-2\int v^l\langle \nabla\eta, \nabla \left(\eta v^l\right)\rangle
-l\int \eta^2v^{2l-1} \Delta v-\f{l-1}{l}\int \eta^2|\nabla v^l|^2.
\end{align}
We have 
\begin{align*}
-\f{l-1}{l}\int \eta^2|\nabla v^l|^2
&=-\f{l-1}{l}\int |\nabla(\eta v^l)-v^l\nabla \eta|^2\\
&=-\f{l-1}{l}\int |\nabla(\eta v^l)|^2+\left(2-\f{2}{l}\right)\int v^l\langle \nabla \eta, \nabla(\eta v^l)\rangle
-\f{l-1}{l}\int v^{2l}|\nabla \eta|^2.
\end{align*}
Inserting the above equality into \eqref{grad-squar-term}, using $-2\langle a\nabla \eta,b\nabla \left(\eta v^l\right)\rangle\leq a^2|\nabla \eta|^2+b^2|\nabla \left(\eta v^l\right)|^2$ with $a=\frac{v^l}{\sqrt{l(l-1)}}$ and $b=\frac{\sqrt{l-1}}{\sqrt{l}}$, and  (\ref{Bochner formula1}) gives
\begin{align}
\begin{split}\label{key-ineq}
\int|\nabla(\eta v^l)|^2
&=-\int v^{2l}\eta\Delta\eta+ \f{l+1}{l}\int v^{2l}|\nabla \eta|^2-\f{2}{l}\int v^l\langle \nabla\eta, \nabla \left(\eta v^l\right)\rangle
\\
&\ \ \ 
-l\int \eta^2v^{2l-1} \Delta v-\f{l-1}{l}\int |\nabla(\eta v^l)|^2\\
&\leq -\int v^{2l}\eta\Delta\eta + \f{l+1}{l}\int v^{2l}|\nabla \eta|^2
+\f{l-1}{l}\int |\nabla(\eta v^l)|^2+\f{1}{l(l-1)}\int v^{2l}|\nabla\eta|^2\\
&\ \ \ \ -\frac{2l}{n}\int \eta^2v^{2l+1}+2l\int \eta^2v^{2l-1}\langle \nabla h, \nabla v\rangle+2l\int \eta^2v^{2l}\rho_-
-\f{l-1}{l}\int |\nabla(\eta v^l)|^2\\
&=
-\int v^{2l}\eta\Delta\eta + \f{l}{l-1}\int v^{2l}|\nabla \eta|^2
\\
& \quad
-\frac{2l}{n}\int \eta^2v^{2l+1}+2l\int \eta^2v^{2l-1}\langle \nabla h, \nabla v\rangle+2l\int \eta^2v^{2l}\rho_-.
\end{split}
\end{align}

We infer from $\Delta h=-v$ and $|\nabla h|\leq v^{\f{1}{2}}$
\begin{align*}
2l\int \eta^2v^{2l-1}\langle \nabla h, \nabla v\rangle
&=\int\langle \eta^2\nabla h, \nabla v^{2l} \rangle
=-\int (\Delta h)\eta^2v^{2l}-2\int\langle\nabla h, \eta\nabla \eta \rangle v^{2l}\\
&\leq \int\eta^2v^{2l+1} + \int \eta^2v^{2l+1} + \int v^{2l}|\nabla \eta|^2.
\end{align*}
Combining (\ref{key-ineq}) with above inequality leads to
\begin{align*}
\int|\nabla(\eta v^l)|^2
&\leq -\int v^{2l}\eta\Delta\eta + \f{2l-1}{l-1}\int v^{2l}|\nabla \eta|^2
+(2-\frac{2l}{n})\int \eta^2v^{2l+1}+2l\int \eta^2v^{2l}\rho_-.
\end{align*}
Let $d(y)=d(x,y)$ be the distance function from $x$. We choose $\eta(y)=\varphi(d(y))$, where  $\varphi: [0, \infty) \rightarrow [0,1]$ satisfying 
 $\varphi(t)\equiv 0$ for $t\geq1$ and $\varphi(t)\equiv1$ for $t\in[0,t_0]$ with $t_0\in (0,1)$, and $\varphi'\leq 0$.
 Thus,
\begin{align*}
|\nabla \eta|=|\varphi'|\quad\text{and}\quad
\Delta \eta=\varphi''+\varphi'\Delta d
          \geq \varphi''+\varphi'\left(\psi+\f{n-1}{d}\right)
          \geq -|\varphi''|- |\varphi'|\psi-\f{n-1}{d}|\varphi'|,
\end{align*}
where $\psi=\left(\Delta d-\f{n-1}{d}\right)_+$.
Hence, for $l\geq n$, we have
\begin{align*}
\int|\nabla(\eta v^l)|^2
&\leq C(n)l\int \left[\left(|\varphi''|+ |\varphi'|\psi+\f{|\varphi'|}{d}\right)\eta v^{2l}+|\varphi'|^2v^{2l}+\eta^2v^{2l}\rho_-\right]\\
&\leq C(n)l\int \left[\left(|\varphi''|+\f{|\varphi'|}{d}\right)\eta v^{2l}+ |\varphi'|\psi\eta v^{2l}+|\varphi'|^2v^{2l}+\eta^2v^{2l}\rho_-\right].
\end{align*}
Choose $\beta=\f{n}{n-2}$. Apply the Sobolev inequality (\ref{normalized-Sobolev-const}) to get
\begin{align}
\label{soblev integral}
\left(\fint\left(\eta^2 v^{2l}\right)^\beta\right)^\frac{1}{\beta}
&\leq\ \  C_s^2\left(B_1(x)\right)\fint\Big|\nabla \left(\eta v^l\right)\Big|^2\nonumber\\
&\leq\ \  C_s^2\left(B_1(x)\right)C(n)l
\fint\Big[\left(|\varphi''|+\f{|\varphi'|}{d}\right)\eta v^{2l}+ |\varphi'|\psi\eta v^{2l}+|\varphi'|^2v^{2l}+\eta^2v^{2l}\rho_-\Big]
\end{align}
To control the $\psi$-term, applying H\"{o}lder's inequality and the Laplacian comparison estimate Theorem~\ref{psi-comp}, we have 
\begin{align}
\label{psi terms}
C(n)l C^2_s\fint|\varphi'|\psi\eta v^{2l}
&\leq C(n)l C^2_s\|\psi\|^*_{2p}\|\eta\varphi'v^{2l}\|^*_{\frac{2p}{2p-1}}
\leq C(n)lC^2_sC(n,p)\left(\kappa(p,1)\right)^{\frac{1}{2}}\|\eta\varphi'v^{2l}\|^*_{\frac{2p}{2p-1}}.
\end{align}
Set $\alpha=\frac{p(n-2)}{n(2p-1)}=\frac{1}{\beta}\frac{p}{2p-1}<1$. Since $\eta\in C^\infty_0(B_1)$ we get
\begin{align}
\begin{split}
\label{eta-varphi'-v^{2l}}
\|\eta\varphi'v^{2l}\|^*_{\frac{2p}{2p-1}}
&=\left[\fint\left(\eta^2v^{2l}\right)^{\alpha\beta}\left(|\varphi'|^2v^{2l}\right)^{\frac{p}{2p-1}}\right]^{\frac{2p-1}{2p}}
\\
&
\leq \left[\left(\fint\left(\eta^2v^{2l}\right)^{\beta}\right)^\alpha \left(\fint\left(|\varphi'|^2v^{2l} \right)^{\frac{np}{np+2p-n}}\right)^{\frac{np+2p-n}{n(2p-1)}}\right]^{\frac{2p-1}{2p}}\\
&\leq \left[\left(\fint\left(\eta^2v^{2l}\right)^{\beta}\right)^\alpha \left(\fint|\varphi'|^2v^{2l} \right)^{\frac{p}{2p-1}}\right]^{\frac{2p-1}{2p}}\\
& \leq \epsilon \left(\fint\left(\eta^2v^{2l}\right)^{\beta}\right)^\frac{1}{\beta}+\frac{1}{4\epsilon}\fint|\varphi'|^2v^{2l},
\end{split}
\end{align}
where we used H\"{o}lder inequality and $\frac{np}{np+2p-n}<1$ due to $p>\frac{n}{2}$ in the second inequality.
By setting $\eps=\left(3C_s^2C(n)l C(n,p)\left(\kappa(p,1)\right)^{1/2}\right)^{-1}$ and inserting (\ref{eta-varphi'-v^{2l}}) into (\ref{psi terms}) we obtain
\begin{equation}\label{psi-term}
C(n)lC^2_s\fint\psi\eta|\varphi'|v^{2l}
\leq \frac{1}{3}\left(\fint\left(\eta^2v^{2l}\right)^{\beta}\right)^\frac{1}{\beta}
   +C(n)l^2C_s^4C^2(n,p)\kappa(p,1)\fint|\varphi'|^2v^{2l}.
\end{equation}
Setting $a=a(n,p)=\frac{2p-n}{2(p-1)}>0$ and using Young's inequality 
\begin{align*}
xy\leq \eps x^b+\eps^{-\frac{b^*}{b}}y^{b^*}, \forall x,y\geq 0, b>1, \frac{1}{b^*}+\frac{1}{b}=1,
\end{align*}
where \[b=\frac{p}{(1-a)(p-1)\beta},\ \ b^*=\frac{p}{(p-1)a}\]
we estimate the $\rho$-term by
\begin{align*}
\fint\eta^2v^{2l}\rho_-
&\leq \|\rho_-\|_p^*\left(\fint
     (\eta^2v^{2l})^\frac{p}{p-1}\right)^\frac{p-1}{p}
\leq \kappa(p,1)\left(\fint\eta^2v^{2l}\right)^{\frac{p-1}{p}a}
     \left(\fint(\eta^2v^{2l})^\beta\right)^{\frac{p-1}{p}(1-a)}\\
&\leq \kappa(p,1)\left[\eps \left(\fint(\eta^2v^{2l})^\beta\right)^\frac{1}{\beta}
+\eps^{-\frac{(1-a)\beta}{a}}\left(\fint\eta^2v^{2l}\right)\right].
\end{align*}
By choosing $\eps=\left(3 C(n)lC_s^2\kappa(p,1)\right)^{-1}$, we obtain
\begin{equation}\label{curvature term}
C_s^2 C(n)l\fint\eta^2v^{2l}\rho_-\\
 \leq \frac{1}{3}\left(\fint(\eta^2v^{2l})^\beta\right)^\frac{1}{\beta}
             + C(n,p)\left(lC_s^2\kappa(p,1)\right)^{\frac{2p}{2p-n}}\left(\fint\eta^2v^{2l}\right).
\end{equation}
Inserting (\ref{psi-term}) and (\ref{curvature term}) into (\ref{soblev integral}) gives
\begin{align}
\label{iteration-relationship}
\left(\fint\left(\eta^2 v^{2l}\right)^\beta\right)^\frac{1}{\beta}
&\leq 3C(n)lC_s^2\Big[\fint\left(|\varphi''|+\frac{|\varphi'|}{d}\right)\eta v^{2l}
+\left(1+C^2(n,p)lC_s^2\kappa(p,1)\right) \fint|\varphi'|^2v^{2l}\Big]\nonumber\\
&\ \ \ \ +C(n,p)\left(lC_s^2\kappa(p,1)\right)^\frac{2p}{2p-n}\left(\fint\eta^2v^{2l}\right).
\end{align}
Define $l=\f{\beta^i}{2}\geq n$, $r_i=\tau-\sum^i_{j=0}2^{-j-1}\delta$ with $\tau\in[\f{1}{2}+\delta,1]$ and $\delta\in(0,\f{1}{2}]$, $B_{i}:=B_{r_i}(x)$. Choose cut-off functions $\eta_i=\varphi_i(d)\in C^\infty_0\left(B_{i}\right)$ such that
\[
\eta_i\equiv 1\ \ \ \text{on}\ B_{i+1};\ \ |\varphi'_i|\leq 2^{i+1}, \ \ |\varphi''_i|\leq 2^{2i+2}.
\]
Substituting $\eta_i$ into (\ref{iteration-relationship}) gives
\begin{align*}
\|v\|^*_{\beta^{i+1},B_{i+1}}
&\leq \Big[C(n)\beta^iC_s^2\left[2+\left(1+C^2(n,p)\beta^iC_s^2\kappa(p,1)\right)\right]4^{i+1}
+C(n,p)\left(\beta^iC_s^2\kappa(p,1)\right)^\frac{2p}{2p-n}\Big]^\f{1}{\beta^i} 
\|v\|^*_{\beta^i,B_{i}}\\
&\leq \left[C(n,p)2^{2i}\left(\beta^q\right)^i \left(C_s^2\left(1+C_s^2\kappa(p,1)\right)
   +\left(C_s^2\kappa(p,1)\right)^\frac{2p}{2p-n}\right)\right]^\f{1}{\beta^i}\|v\|^*_{\beta^i, B_{i}}\\
&\leq \left(4\cdot \beta^q\right)^\f{i}{\beta^i}\left(C(n,p)A\right)^\f{1}{\beta^i}\|v\|^*_{\beta^i,B_{i}},
\end{align*}
where
\[
q=\max\left\{2, \frac{2p}{2p-n}\right\}, \ \ \ A=C_s^2\left(1+C_s^2\kappa(p,1)\right)+\left(C_s^2\kappa(p,1)\right)^\frac{2p}{2p-n}.
\]
Iterating from $i_0$ such that $n+1>\beta^{i_0}\geq n$ to $\infty$, since $\sum_{i=0}^\infty \f{1}{\beta^i}=\f{n}{2}$ and $\sum_{i=0}^\infty \f{i}{\beta^i}$ are finite, we obtain for $k<n$
\begin{align*}
\|v\|^*_{\infty, B_{\tau-\delta}}
&\leq C(n,p)A^\f{n}{2}\|v\|^*_{n,B_\tau}
\leq  C(n,p)A^{\f{n}{2}}\left(\|v\|^*_{k,B_\tau}\right)^\f{k}{n}\left(\|v\|^*_{\infty, B_\tau}\right)^{1-\f{k}{n}}.
\end{align*}
Let
\[
\delta_i=2^{-i-3}, \tau_0=\f{1}{2}+\delta_0,\quad \tau_{i+1}=\tau_i+\delta_i,\quad i=0,1,2,\ldots.
\]
Iterating from $0$ to $i$, we have
\begin{align*}
\|v\|^*_{\infty, B_{\tau_0-\delta_0}}
&\leq C(n,p)A^{\f{n}{2}}\left(\|v\|^*_{k, B_{\tau_0}}\right)^\f{k}{n}\left(\|v\|^*_{\infty, B_{\tau_0}}\right)^{1-\f{k}{n}}
\\
&
\leq \prod^i_{j=0}\left[C(n,p)A^{\f{n}{2}}\left(\|v\|^*_{k, B_{\tau_j}}\right)^\f{k}{n}\right]^{\left({1-\f{k}{n}}\right)^j}\left(\|v\|^*_{\infty, B_{\tau_{i}}}\right)^{\left({1-\f{k}{n}}\right)^{i+1}}.
\end{align*}
Let $i\to\infty$, then $\left({1-\f{k}{n}}\right)^{i+1}\to 0$ and $\sum_{j=0}^\infty \left({1-\f{k}{n}}\right)^j=\f{n}{k}$. Hence,
\begin{align*}
\|v\|^*_{\infty, B_\f{1}{2}}\leq \left(C(n,p)A^{\f{n}{2}}\right)^\f{n}{k}\|v\|^*_{k, B_{\tau_\infty}}\leq C^\f{n}{k}(n,p)A^{\f{n^2}{2k}}\|v\|^*_{k, B_\f{7}{8}}.
\end{align*}
Thus, we get
\begin{align}
\begin{split}\label{iretation result}
\sup\limits_{B_\f{1}{2}}|\nabla h|^2\leq \sup\limits_{B_\f{1}{2}}v\leq C(n,p)A^\f{n^2}{2}\fint_{B_\f{7}{8}} v
\end{split}
\end{align}
by setting $k=1$ and using $v=|\nabla h|^2+\lambda$.
Choosing $\eta\in C_0^\infty(B_1)$ with $\eta\equiv 1$ in $B_\f{7}{8}$ and $|\nabla \eta|\leq 8$, we have
\begin{align*}
\fint \eta^2 |\nabla h|^2
=\fint\eta^2 \left(-\Delta h - \lambda\right)
\leq2\fint\eta\langle \nabla h, \nabla \eta\rangle
\leq\frac{1}{2}\fint \eta^2 |\nabla h|^2 +2 \fint |\nabla \eta|^2
\leq \frac{1}{2}\fint \eta^2 |\nabla h|^2 +128.
\end{align*}
Thus,
\begin{align*}\begin{split}\label{gradient-h-esti}
\fint \eta^2 |\nabla h|^2\leq 256,
\end{split}\end{align*}
and hence
\begin{align*}
\fint_{B_\f{7}{8}}(v-\lambda)
=\fint_{B_\f{7}{8}}|\nabla h|^2
\leq \f{\vol\left(B_1\right)}{\vol\left(B_\f{7}{8}\right)}\fint_{B_1} \eta^2 |\nabla h|^2
\leq 256 \f{\vol\left(B_1\right)}{\vol\left(B_\f{7}{8}\right)}.
\end{align*}
Inserting above inequality into (\ref{iretation result}) gives
\begin{align*}
\sup\limits_{B_\f{1}{2}}|\nabla h|^2
&\leq C(n,p)\left(\f{\vol\left(B_1\right)}{\vol\left(B_\f{7}{8}\right)}+\lambda \right)
\left[C_s^2\left(B_1\right)\left(1+C_s^2\left(B_1\right)\kappa(p,1)\right)+\left(C_s^2\left(B_1\right)\kappa(p,1)\right)^\frac{2p}{2p-n}\right]^{\f{n^2}{2}}.
\end{align*}
The desired result can be obtained by scaling.
\end{proof}
 With volume doubling the following upper bound for the first Dirichlet eigenvalue of the ball follows easily from a simple test-function argument.
\begin{lemma}\label{lemma:upper}
Let $(M^n, g)$ be a complete Riemannian manifold. 
Given $p>\frac{n}{2}, D>0$, there exists $\epsilon=\epsilon(n,p)>0$ and $C=C(n)>0$ such that if  $\kappa(p,D)<\epsilon$, then for any $x\in M, \ 0 <r \le D$
\[
\lambda_1(B_r(x))\leq C\, r^{-2}.
\]

\end{lemma}
This is well known to the experts. For completeness we present a proof. 
\begin{proof} Let $\varphi:[0,r] \rightarrow [0,1]$ be the cut off function with $\varphi |_{[0, \tfrac r4]} =1, \varphi |_{[\tfrac 34 r, r]} =0$, and $|\varphi'| \le \tfrac 2r$. Let $d(y) = d(x,y)$ be the distance function from $x$. Using the test function $f(y) = \varphi (d(y))$ we have \[ \lambda_1(B_r(x))\leq \frac{\int_{B_r(x)} |\nabla f|^2}{\int_{B_r(x)} f^2} \le 4r^{-2} \frac{\vol (B_{r})}{\vol (B_{\tfrac{r}{4}})} \le 2 \cdot 4^{n+1} r^{-2}.
\]
Here in the last step we use the volume doubling estimate Theorem~\ref{volume-comp}. 
\end{proof}

Now using the local Sobolev estimate Theorem~\ref{thm:local-Sobolev}, Theorem~\ref{grad-esti-u} gives the following Harnack estimate for the first eigenfunction of $\Omega$ on a ball. 
\begin{theorem}\label{loc-grad-est}
Let $p>\f{n}{2}$ and $\Omega\subset M$ be a bounded with $\d \Omega\neq \varnothing$. Assume that $u_1$ is the positive first Dirichlet eigenfunction of $\Omega$ with $\diam \Omega = D$.
For $B_r(x)\subset\Omega$ satisfying $B_r(x)\cap\d \Omega=\varnothing$, there exist $\eps=\eps(n,p)>0$ and $C=C(n,p)>0$ such that if $\kappa(p,D)\leq \eps$, we have
\begin{align}
\begin{split}\label{local harnack ineq}
\sup\limits_{B_{r/2}(x)}u_1\leq C(n,p)\inf\limits_{B_{r/2}(x)}u_1.
\end{split}
\end{align}
\end{theorem}

\begin{proof}
 Let $\lambda_1>0$ be the first Dirichlet eigenvalue of $\Omega$. Then \[\Delta u_1=-\lambda_1 u_1. \] 
 By domain monotonicity of eigenvalues, $\lambda_1 \le \lambda_1 (B_r(x))$. 

Combining Lemma~\ref{lemma:upper}, volume doubling (cf.~Theorem~\ref{volume-comp}), Theorem \ref{grad-esti-u}, and Theorem~\ref{thm:local-Sobolev}, we obtain
\begin{align}\begin{split}\label{Grad-esti-lnu_1}
\sup\limits_{B_{\f{r}{2}}(x)}|\nabla \ln u_1|^2\leq C_1(n,p)r^{-2}.
\end{split}\end{align}

For any $y,z \in B_{\f{r}{2}}(x)$,  let $\gamma(s):[0,l]\rightarrow \Omega$ be a distance minimizing geodesic joint $y$ and $z$ parametrized by arc length. By (\ref{Grad-esti-lnu_1}), we have
\begin{align*}
\ln \frac{u_1(z)}{u_1(y)} & =\int^l_0 \frac{d}{ds}\ln u_1\left(\gamma(s)\right)\drm s
                     =\int^l_0\langle \nabla \ln u_1, \gamma'(s)\rangle ds
                     \leq \sqrt{C_1(n,p)}r^{-1}l
                     \leq \sqrt{C_1(n,p)}.
\end{align*}
where we used $l\leq r$. Hence,
\begin{equation}\label{harnack-interior-pts}
u_1(z)\leq \exp \left(\sqrt{C_1(n,p)}\right)u_1(y).
\end{equation}
According to the arbitrariness of $y$ and $z$, we have
\[
\sup\limits_{B_{r/2}(x)}u_1\leq C(n,p)\inf\limits_{B_{r/2}(x)}u_1.
\]
This completes the proof of the theorem.
\end{proof}

\section{Global Harnack inequality and the Fundamental Gap Estimate}\label{glob_harn_fundgap}
In this section we derive a global Harnack inequality for the first eigenfunction, therefore getting an estimate on the fundamental gap. The proofs are the same as in  \cite{Oden-Sung-Wang99}, so we omit the proofs. 

Let $d(x)=d(x,\partial\Omega)$ and $\Omega_t=\{x\in \Omega|d(x)\geq t\}$. Denote  $T(\partial\Omega,t) = \Omega\setminus \Omega_t$.

With the assumption that $\Omega$ is $(R,H,K)$-regular, namely the geometry of the boundary and near the   boundary are well controlled, 
and the local Harnack estimate Theorem~\ref{loc-grad-est} proved in the last section, the argument in \cite[Section 2]{Oden-Sung-Wang99} shows that the first positive eigenfunction $u_1$ of $\Omega$ is uniformly bounded from below in $\Omega_\delta$ and is quasi-isometric to $d(x)$ near the boundary, where $\delta=\delta(n,K,H,R)$ is given by \eqref{deltaref}. More precisely, there exist $\epsilon(n,p)>0,\ \delta(n, H,K,R)>0$, $C_1(n,p,H,K,R,D)>0,\ C_2(n,H,K,R)>0$ such that if $\kappa(p,D) < \epsilon$ and we normalize $u_1$ so that $\sup_\Omega u_1 =1$, then
\begin{eqnarray*}
u_1 (x) \ge C_1(n,p,H,K,R,D) \ \ & x \in \Omega_\delta \\
\frac{2}{3\delta}C_1(n,p,H,K,R,D) \, d(x) 
 \le u_1 (x) \le C_2(n,H,K,R)\, d(x) \ \ & x \in   T(\partial\Omega,\delta).
\end{eqnarray*}

The proofs of these estimates follow as in \cite{Oden-Sung-Wang99}, see Lemmas~\ref{u-uper-bdd-rho}-\ref{u-lower-bdd-rho}. From these estimates we can derive the global Harnack inequality that we need, Theorem~\ref{Harnack-inequality-esti}. The details are provided in the \hyperref[appendix]{appendix} for completion.
 Namely, we have 
\begin{theorem}\label{Harnack-inequality-esti}
Let $p>\f{n}{2}$, $D>0$, $H, K\geq0$. There exist  explicitly computable $R_0=R_0(H,K)>0$ and $\eps=\eps(n,p)>0$ such that the following holds ($R_0$ is given in Remark~\ref{R0-condition}).
If
\[
\kappa(p,D)\leq \eps,
\]
then for any $0<R\leq R_0$ there exists a computable constant $C_3=C_3(n, p, D, R, H,K)>0$ such that the first positive  Dirichlet eigenfunction $u_1$ of an $(R,H,K)$-regular domain $\Omega\subset M^n$
satisfies
\[
u_1(x)\leq C_3 u_1(y)
\]
for all $x, y\in \Omega$ with $0<d(x)\leq 2d(y)$.
\end{theorem}
\begin{remark}\label{R0-condition}
The constant $R_0\in(0,1)$ is chosen to satisfy
$$\sqrt{K}\tan(R_0\sqrt{K})\leq \f{H}{2}+\f{1}{2},\quad \frac{H}{\sqrt{K}}\tan(R_0\sqrt{K})\leq \f{1}{2},\quad\text{and}\quad 
\frac{H}{\sqrt{K}}\tanh(R_0\sqrt{K})\leq \f{1}{2}.$$
\end{remark}

We are now in the position to prove Theorem~\ref{main1}. Recall the following fundamental gap estimate obtained in \cite{Oden-Sung-Wang99} which will serve as the basis for the proof.

\begin{theorem}\cite[Theorem 1.2]{Oden-Sung-Wang99}\label{OSW-gap-est}
Let $M$ be an $n$-dimensional compact Riemannian manifold with $\d M\neq \varnothing$ and $\diam M\leq D$. 
Suppose 
\begin{itemize}
\item[-] the first nonzero Neumann eigenvalue $\eta_1$ of $M_{R/2}$ is bounded below by $C_\eta>0$,
\item[-] the volume doubling property holds on $M$ with volume doubling constant $C_0$,
\item[-] the weak Neumann-Poincar{\'e} inequality holds on all balls $B=B_r(x)$ with $2B\subset\Omega$, $2B\cap\partial\Omega=\emptyset$ with weak Neumann-Poincar\'{e} constant $C_Pr$, where $C_P$ is independent of $x$ and $r$,
\item[-] $\d M$ satisfies the interior rolling $R$-ball condition,
\item[-] the first Dirichlet eigenfunction on $M$, $u_1$, satisfies $u_1(x)\leq C_3 u_1(y)$ for all $x,y\in M$ with $0<d(x,\d M)\leq 2 d(y,\d M)$.
\end{itemize}
 Then
\[
\lambda_2-\lambda_1\geq C(C_0, C_P, C_3, R, C_\eta, D).
\]
\end{theorem}

We apply Theorem~\ref{OSW-gap-est} to our setting to obtain our main result. In the above sections we showed that the volume doubling property, local Neumann-Poincar\'{e} inequality \eqref{weakNP}, and Harnack inequality for the first Dirichlet eigenfunction \eqref{Harnack-inequality-esti} are satisfied. What is left is deriving a lower bound for the first Neumann eigenvalue of $\Omega_{R/2}$ where $\Omega$ is $(R,H,K)$-regular. As in \cite[Page 3543]{Oden-Sung-Wang99} we can show 
 $\Omega_{R/2}$ is $\left(R/4, 2(\sqrt{K}+H), K\right)$-regular. By (\ref{diam-est}) $\diam \Omega_{R/2} \le C(n,H,K,R,D)$. Applying Proposition~\ref{Neumann-eigen} then gives
\[
\eta_1\geq C(n,p,R,D, H, K).
\]
 Then we obtain the desired fundamental gap estimate  by applying Theorem \ref{OSW-gap-est}.

 \bibliographystyle{alpha}

 \newpage

\begin{appendix}
\section{Global Harnack inequality of the first Dirichlet eigenfunction}\label{Global-Harnack-inequality}\label{appendix}
In this appendix we include, for completeness, the estimates needed in Section~\ref{glob_harn_fundgap}, which follow directly from the proofs in \cite[Section 2]{Oden-Sung-Wang99}. The main goal is to prove Theorem~\ref{Harnack-inequality-esti}.

Let $d(x)=d(x,\partial\Omega)$ and $\Omega_t=\{x\in \Omega|d(x)\geq t\}$. Denote  $T(\partial\Omega,t) = \Omega\setminus \Omega_t$. 
To obtain Theorem~\ref{Harnack-inequality-esti}, on $\Omega_\delta$ we will use the local Harnack inequality, while on the tubular neighborhood of $\partial\Omega$, we will use that $d(x)$ and the first Dirichlet eigenfunction are comparable in that region.
Essential for the comparison is a bound on $\Delta d(x)$.
The Laplacian comparison estimate for $d$ has the following form.
\begin{proposition}\label{lap-com-thm}
Let $\Omega\subset M^n$ be a $(R,H,K)$-regular domain.
Then we have 
\[
-\frac{1}{2\delta}\leq\Delta d\leq \frac{1}{2\delta}\quad\text{and}\quad\ \ \Delta d^2\geq 1
\]
on $T(\partial\Omega,\delta)$ with
\begin{equation}\label{deltaref}
\delta=\f{1}{4(n-1)(\sqrt{K}+H+\frac{1}{R})}.
\end{equation}
\end{proposition}

The proof of Prop.~\ref{lap-com-thm} follows from the following lemma.

\begin{lemma}\cite[Lemma 3.2.3]{Oden94}  \cite[Lemma 2.1]{Oden-Sung-Wang99}\label{OSW-lap-com}
Let $N$ be a compact Reimannnian manifold with boundary $\partial N$, and $a, b$ be two constants such that $b\geq a>0$, $b\geq1$ and the sectional curvature satisfies that  $-a^2\leq K_N\leq b^2$. Then $d(x)=d(x,\d N)$ is at least $C^2$ on $N\setminus N_\delta$ if $\d N$ is $C^{k-1}$ and
\begin{align*}
\sum_{i=1}^{n-1}\frac{-b\tan(bd)-k_i(y)}{1-\frac{k_i(y)}{b}\tan(bd)}\leq\Delta d
\leq \sum_{i=1}^{n-1}\frac{-a\tanh(ad)-k_i(y)}{1-\frac{k_i(y)}{a}\tanh(ad)},
\end{align*}
where $k_i(y)$ is the $i$-th principal curvature with respect to the inward pointing unit normal at the unique point $y\in\d N$ such that $d(x)=d(x,y)$ and $\delta$ is as in \eqref{deltaref}.
\end{lemma}

\begin{proof}[Proof of Proposition \ref{lap-com-thm}]
Since $\delta<R$, $T(\partial \Omega, \delta)\subset T(\partial \Omega, R)$. For any $x\in T(\partial \Omega, \delta)$ and $y\in \d \Omega$ such that $d(x)=d(x,y)$, using Lemma \ref{OSW-lap-com},  we have
\begin{align}\label{upp-bdd-lap}
\Delta d\geq \sum^{n-1}_{i=1}
\f{-\sqrt{K}\tan(\sqrt{K}d)-k_i(y)}{1-\f{k_i(y)}{\sqrt{K}}\tan(\sqrt{K}d)} 
\quad\text{and}\quad
\Delta d
\leq \sum^{n-1}_{i=1}\frac{-\sqrt{K}\tanh(\sqrt{K}d)-k_i(y)}{1-\frac{k_i(y)}{\sqrt{K}}\tanh(\sqrt{K}d)},
\end{align}
where $k_i(y)$ is the principal curvature at $y$ with respect to the unit inward-pointing normal vector.
Since $\II\leq H$, we have
$k_i(y)\leq H$, and since $R\in (0, R_0]$
\begin{align}\label{Denominator-est2}
1-\frac{k_i(y)}{\sqrt{K}}\tan\left(\sqrt{K}d\right)\geq 1-\frac{H}{\sqrt{K}}\tan\left(\sqrt{K}R\right)\geq\frac{1}{2}
\end{align}
and
\begin{align*}
1-\frac{k_i(y)}{\sqrt{K}}\tanh\left(\sqrt{K}d\right)\geq 1-\frac{H}{\sqrt{K}}\tanh\left(\sqrt{K}R\right)\geq\frac{1}{2}.
\end{align*}
Combining (\ref{Denominator-est2}) with (\ref{upp-bdd-lap}) and $-1<\tanh x<1$, we obtain
\begin{equation*}
\Delta d
\leq 2(n-1)\left(\sqrt{K}+H\right)
\leq \frac{1}{2\delta}.
\end{equation*}
Since $k_i(y)\leq H$, and $\sqrt{K}d\leq \sqrt{K}\delta\leq \frac{\pi}{4}$,
we have
\begin{equation*}
\Delta d
\geq \sum^{n-1}_{i=1}
\frac{-\sqrt{K}\tan(\sqrt{K}d)-k_i(y)}{1-\frac{k_i(y)}{\sqrt{K}}\tan(\sqrt{K}d)}
\geq -2\sum^{n-1}_{i=1}(\sqrt{K}+H) \\
\geq-\frac{1}{2\delta}.
\end{equation*}
Hence, we obtain
\[
\Delta d^2=2d\Delta d+2|\nabla d|^2\geq -\frac{d}{\delta} +2\geq 1.
\]
\end{proof}

Proposition \ref{lap-com-thm} enables us to replace $\sup\limits_{\Omega} u_1$ by $d$ near the boundary, as the following lemmas show.
\begin{lemma}\label{u-uper-bdd-rho}
Let $\Omega\subset M^n$ be a $(R,H,K)$-regular domain. Suppose $u_1$ is the first Dirichlet eigenfunction on $\Omega$. Then there exists $\eps=\eps(n,p)>0$ such that if $\kappa(p, D)<\eps$, we have 
\[
u_1(x)\leq \frac{4 C(n)}{\delta^3}d(x), \quad x\in T(\partial\Omega, \delta),
\]
where $C(n)$ is the constant from Lemma \ref{lemma:upper}.
If $x_0\in\overline\Omega$ with $u_1(x_0)=\sup\limits_{x\in \Omega}u_1(x)$, we have
$d(x_0)\geq  \f{\delta^3}{4 C(n)}$.
\end{lemma}
\begin{proof}
W.l.o.g., $\sup\limits_\Omega u_1=1$. Let $\beta_1$ and $\beta_2$ be positive constants to be chosen later. By Lemma \ref{OSW-lap-com}, we have on $\Omega\setminus \Omega_\delta$
\begin{align}\label{Delta-u-rho-geq}
\Delta\left(u_1-\beta_2(d-\beta_1d^2)\right)
&=-\lambda_1 u_1-\beta_2\left(\Delta d-\beta_1\Delta d^2\right)
\geq -\lambda_1-\beta_2\left(\f{1}{2\delta}-\beta_1\right).
\end{align}

Since $\delta<R$ and $\partial\Omega$ satisfies the interior rolling $R$-ball condition, there exists a geodesic ball $B_{\delta}(x)\subset \Omega$. 
We infer from Lemma \ref{lemma:upper} that there exists  $\epsilon=\epsilon(n,p)$ such that if $\kappa(p,D)<\epsilon$ then $\lambda_1(\Omega)\leq \lambda_1(B_\delta(x))\leq \mu$ with 
\begin{equation}
\mu=C(n)\delta^{-2}.  \label{mu}
\end{equation}

Choose $\beta_1=\frac{3}{4\delta}$ and $\beta_2=\f{4\mu}{\delta}$. Since $R\leq R_0<1$, then $\delta<1$, so from (\ref{Delta-u-rho-geq}) we get
\begin{equation}\label{lap-geq}
\Delta\left(u_1-\beta_2(d-\beta_1d^2)\right)
\geq -\lambda_1-\frac{4\mu}{\delta}\left(\f{1}{2\delta}-\frac{3}{4\delta}\right)
\geq -\mu+\frac{\mu}{\delta^2}\geq 0.
\end{equation}
Since $\mu> 1$, we have on $\d \Omega_\delta$
\begin{equation}\label{bdy-cond1}
u_1-\beta_2(d-\beta_1d^2)
= u_1-\frac{4\mu}{\delta}\left(\delta-\frac{3}{4\delta}\delta^2\right)
\leq1-\mu\leq0.
\end{equation}

(\ref{lap-geq}), (\ref{bdy-cond1}), $u_1-\beta_2(d-\beta_1d^2)=0$ on $\d \Omega$, and the maximum principle imply 
\begin{align}\label{u-up-bdd-by-rho}
u_1\leq \beta_2(d-\beta_1d^2)\leq \beta_2d=\frac{4\mu}{\delta}d
\end{align}
on $T(\partial\Omega, \delta)\cap\overline\Omega$.
Since $u_1(x_0)=\sup\limits_\Omega u_1$, (\ref{u-up-bdd-by-rho}) yields
\[
d(x_0)\geq \frac{\delta}{4\mu}u_1(x_0).\qedhere
\]
\end{proof}

In order to bound $u_1$ from below using the maximum principle, it is necessary to bound $u_1$ from below on $\{d=t\}$ for suitable $t$.
\begin{lemma}\label{infimum-of-u}
Assume the setting of Theorem~\ref{Harnack-inequality-esti} and let $u_1$ be the first Dirichlet eigenfunction on $\Omega$. Let $u_1(x_0)=\sup\limits_\Omega u_1=1$.
Then on $\overline{\Omega_{\delta/4\mu }}$ we have
	\[
u_1\geq 
C(n,p)^{-\frac{8\mu C_{H,K,\delta/{4\mu}}^{n-1}D}{\delta}},
\]
where $C(n,p)$ is given in Theorem~\ref{loc-grad-est}, $\mu$ by \eqref{mu}, and $C_{H,K,\delta/{4\mu}}$ by 
\begin{equation}\label{C(H,K,t)}
C_{H,K,\delta/{4\mu}}:=2\left(\frac{K}{H}+\frac{H}{K}\right)\cosh\left(\frac{\sqrt{K}\delta}{2\mu}\right).
\end{equation}
\end{lemma}
\begin{proof}
Let $t=\frac{\delta}{4\mu}$. For any $x\in \Omega_t$, $B_t(x)\subset \Omega$. Theorem \ref{loc-grad-est} yields that
\begin{equation}\label{local harnack ineq1}
\sup\limits_{B_{t/2}(x)}u_1
\leq C(n,p)\inf\limits_{B_{t/2}(x)}u_1.
\end{equation}

Since $\delta<R$, $\mu>\frac{4\pi^2}{\delta^2}\geq4\pi^2$,  $t<\frac{R}{2}$. Hence, $\Omega_t$ is connected. It follows as in \cite[Lemma~3.2.7,p.~60]{Oden94} that we have
\[
\diam(\Omega_t)\leq C_{H,K,t}^{n-1}D.
\]
The closed set $\overline{\Omega_t}$ is compact, so the Hopf-Rinow Theorem implies that through the two points $y,z\in \Omega_t$ satisfying $u_1(y)=\sup_{\Omega_t} u_1$ and $u_1(z)=\inf_{\Omega_t}u_1$, there exists a minimizing geodesic $\gamma_{yz}\subset \Omega_t$ joining $y$ and $z$. Thus the length of $\gamma_{yz}$ is at most $\diam(\Omega_t)$. Using (\ref{local harnack ineq1}) repeatedly along $\gamma_{yz}$ gives
\[
1=\sup_{\Omega_t} u_1=u_1(y)\leq C(n,p)^\frac{2\diam(\Omega_t)}{t} u_1(z)=C(n,p)^\frac{2\diam(\Omega_t)}{t}\inf_{\Omega_t}u_1
\leq C(n,p)^\frac{8\mu C_{H,K,t}^{n-1}D}{\delta} \inf_{\Omega_t}u_1.
\qedhere
\]
\end{proof}

\begin{lemma}\label{u-lower-bdd-rho}
Assume the setting of Theorem~\ref{Harnack-inequality-esti} and let $u_1$ be the first Dirichlet eigenfunction on $\Omega$ with $u_1(x_0)=\sup\limits_\Omega u_1=1$. We have
\[
u_1(x)
\geq\frac{2}{3 \delta}C(n,p)^{-\frac{8\mu C_{H,K,\delta/{4\mu}}^{n-1}D}{\delta}} d(x),
\quad x\in T(\partial\Omega,\delta),
\]
where $C_{H,K,\delta/4\mu}$ is given by \eqref{C(H,K,t)}.
\end{lemma}

\begin{proof}
Let $\gamma_1>0$, $\gamma_2>0$ be two constants to be chosen later. We infer from Proposition~\ref{lap-com-thm} that
\begin{equation}\label{boundary case for u}
\Delta \left(u_1-\gamma_2\left(d+\gamma_1 d^2\right)\right)
\leq-\lambda_1u_1-\gamma_2\left(-\f{1}{2\delta}+\gamma_1\right)
\end{equation}
holds on  $T(\partial\Omega,\delta)$.
Hence, the right hand of (\ref{boundary case for u}) is less than or equal to zero if we choose $\gamma_1=\frac{1}{2\delta}$. We will apply the maximum principle on $\d T(\partial\Omega,\delta)$.
By Lemma \ref{infimum-of-u}, if 
$\gamma_2=\frac{2}{3 \delta}C(n,p)^{-\frac{8\mu C_{H,K,\delta/{4\mu}}^{n-1}D}{\delta}}$, 
we have for $x\in\d \Omega_\delta$,
	\begin{align*}
u_1(x)-\gamma_2\left(d+\gamma_1 d^2\right)
&\geq C(n,p)^{-\frac{8\mu C_{H,K,{\delta/4\mu}}^{n-1}D}{\delta}}
-\gamma_2\left(\delta+\frac{1}{2\delta}\delta^2\right)\\
&= C(n,p)^{-\frac{8\mu C_{H,K,{\delta/4\mu}}^{n-1}D}{\delta}}-\frac{3}{2}\gamma_2\delta\\
&= 0.
\end{align*}
For any $x\in \d \Omega$, $u_1(x)-\gamma_2\left(d+\gamma_1 d^2\right)=0$.
Hence, the maximum principle yields
\[u_1(x)\geq\gamma_2\left(d+\gamma_1 d^2\right)
\geq \frac{2}{3 \delta}C(n,p)^{-\frac{8\mu C_{H,K,\delta/{4\mu}}^{n-1}D}{\delta}}d(x).
\qedhere\]
\end{proof}
The above lemmas yield the desired Harnack inequality for the first Dirichlet eigenfunction.

\begin{proof}[Proof of Theorem \ref{Harnack-inequality-esti}]
For $x\in T(\d \Omega, \delta)\cap\Omega$, we infer from Lemma~\ref{u-uper-bdd-rho} and Lemma~\ref{u-lower-bdd-rho}
\begin{align}\label{boundary-case}
\frac{2}{3 \delta}C_0d(x)\leq u_1(x)\leq \frac{4\mu}{\delta}d(x),
\end{align}
where 
$C_0=C(n,p)^{-\frac{8\mu C_{H,K,\delta/4\mu}^{n-1}D}{\delta}}$.
Since $\delta>\frac{\delta}{4\mu}$, Lemma~\ref{infimum-of-u} implies 
\begin{align}\begin{split}\label{interior-case}
C_0\leq u_1(x)\leq 1, \quad x\in \Omega_\delta.
\end{split}\end{align}
For $x, y\in \Omega$ with $0<d(x)\leq 2d(y)$, we distinguish several cases.
\\
\textbf{Case 1}. $x, y\in T(\partial\Omega,\delta)$. (\ref{boundary-case}) yields
\[
u_1(x)\leq \frac{4\mu}{\delta}d(x)\leq \frac{8\mu}{\delta}d(y)\leq \frac{12\mu}{C_0}u_1(y).
\]
\textbf{Case 2}. $x, y\in \Omega_\delta$.  (\ref{interior-case}) yields
\[
u_1(x)\leq 1\leq C_0^{-1}u_1(y).
\]
\textbf{Case 3}. $x\in T(\partial\Omega,\delta)\cap\Omega$ and $y\in \Omega_\delta$. (\ref{boundary-case}) and (\ref{interior-case}) imply
\[
u_1(x)\leq \frac{4\mu}{\delta}d(x)\leq 4\mu\leq \frac{4\mu}{C_0}u_1(y).  
\]
\textbf{Case 4}. $x\in \Omega_\delta$ and $y\in T(\partial\Omega,\delta)\cap\Omega$. We infer from (\ref{boundary-case}) and (\ref{interior-case})
\[
u_1(x)\leq 1\leq \frac{d(x)}{\delta}\leq \frac{2d(y)}{\delta}
\leq \frac{2}{\delta}\cdot \frac{3\delta}{2C_0}u_1(y)=\frac{3}{C_0}u_1(y).
\]
The claim follows by setting $C_3=12\mu C_0^{-1}$.
  \end{proof}
 \end{appendix}
\end{document}